\documentclass[letterpaper, 10 pt, conference]{ieeeconf}  % Comment this line out
\IEEEoverridecommandlockouts                              % This command is only
\usepackage{amsmath}
\usepackage{graphicx}
\usepackage[colorinlistoftodos]{todonotes}
\usepackage[colorlinks=true, allcolors=blue]{hyperref}
\usepackage{amssymb}
\usepackage{cite}

\usepackage{mathtools}

\newtheorem{corollary}{Corollary}
\newtheorem{lemma}{Lemma}
\newtheorem{theorem}{Theorem}

\title{\LARGE \bf
On the Location of the Minimizer of the Sum of\\ Two Strongly Convex Functions 
}
\author{Kananart Kuwaranancharoen and Shreyas Sundaram % <-this % stops a space
\thanks{This research was supported by NSF CAREER award
1653648.  The authors are with the School of Electrical and Computer Engineering at Purdue University.  Email: {\tt \{kkuwaran,sundara2\}@purdue.edu}.}% <-this % stops a space
}

\begin{document}

\maketitle
\thispagestyle{empty}
\pagestyle{empty}

%%%%%%%%%%%%%%%%%%%%%%%%%%%%%%%%%%%%%%%%%%%%%%%%%%%%%%%%%%%%%%%%%%%%%%%%%%%%%%%%
\begin{abstract}

The problem of finding the minimizer of a sum of convex functions is central to the field of distributed optimization.  Thus, it is of interest to understand how that minimizer is related to the properties of the individual functions in the sum. In this paper, we provide an upper bound on the region containing the minimizer of the sum of two strongly convex functions.   We consider two scenarios with different constraints on the upper bound of the gradients of the functions. In the first scenario, the gradient constraint is imposed on the location of the potential minimizer, while in the second scenario, the gradient constraint is imposed on a given convex set in which the minimizers of two original functions are embedded. We characterize the boundaries of the regions containing the minimizer in both scenarios. % and provide the shape of this region when the constraint set is given by a circle or a box.

\end{abstract}

\section{Introduction}
The problem of distributed optimization arises in a variety of applications, including machine learning \cite{Ma09, Shalev11, Boyd11ADMM, Sayed14}, control of large-scale systems \cite{DistOptGrid14, Li2011optimal}, and cooperative robotic systems \cite{Schwager09,hosseini2013online, Montijano14,hosseini2014online,Zhu15}. In such problems, each node in a network has access to a local convex function (e.g., representing certain data available at that node), and all nodes are required to calculate the minimizer of the sum of the local functions.  There is a significant literature on distributed algorithms that allow the nodes to achieve this objective \cite{JNT-DPB-MA:86, Nedic10constrained,BJ-MR-MJ:09,MZ-SM:12,JW-NE:11, BG-JC:14-tac, AN-AO:15-tac}.  The local functions in the above settings are typically assumed to be private to the nodes.  However, there are certain common assumptions that are made about the characteristics of such functions, including strong convexity and bounds on the gradients (e.g., due to minimization over a convex set).  

In certain applications, it may be of interest to determine a region where the minimizer of the sum of the functions can be located, given only the minimizers of the local functions, their strong convexity parameters, and the bound on their gradients (either at the minimizer or at the boundaries of a convex constraint set).  For example, when the network contains malicious nodes that do not follow the distributed optimization algorithm, one cannot guarantee that all nodes calculate the true minimizer.  Instead, one must settle for algorithms that allow the non-malicious nodes to converge to a certain region \cite{sundaram2016secure,LS-NV:15}.  In such situations, knowing the region where the minimizer can lie would allow us to evaluate the efficacy of such resilient distributed optimization algorithms.  Similarly, suppose that the true functions at some (or all) nodes are not known (e.g., due to noisy data, or if the nodes obfuscate their functions due to privacy concerns).  A key question in such scenarios is to determine how far the minimizer of the sum of the true functions can be from the minimizer calculated from the noisy (or obfuscated) functions.  The region containing all possible minimizers of the sum of functions (calculated using only their local minimizers, convexity parameters, and bound on the gradients) would provide the answer to this question.  

When the local functions $f_i$ at each node $v_i$ are single dimensional (i.e., $f_i: \mathbb{R} \rightarrow \mathbb{R}$), and strongly convex, it is easy to see that the minimizer of the sum of functions must be in the interval bracketed by the smallest and largest minimizers of the local functions.  This is because the gradients of all the functions will have the same sign outside that region, and thus cannot sum to zero.  However, a similar characterization of the region containing the minimizer of multidimensional functions is lacking in the literature, and is significantly more challenging to obtain.  For example, the conjecture that the minimizer of a sum of convex functions is in the convex hull of their local minimizers can be easily disproved via simple examples; consider $f_1(x,y) = x^2 - xy + \frac{1}{2} y^2$ and $f_2(x,y) = x^2 + xy + \frac{1}{2} y^2 -4x -2y$ with minimizers $(0,0)$ and $(2,0)$ respectively, whose sum has minimizer $(1, 1)$. Thus, in this paper, {\bf our goal is to take a step toward characterizing the region containing the minimizer of a sum of strongly convex functions.}  Specifically, we focus on characterizing this region for the sum of {\it two} strongly convex functions under various assumptions on their gradients (as described in the next section).  As we will see, the analysis is significantly complicated even for this scenario.  Nevertheless, we obtain such a region and gain insights that could potentially be leveraged in future work to tackle the sum of multiple functions.

%%%%%%%%%%%%%%%%%%%%%%%%%%%%%%%%%%%%%%%%%%
\section{Notation and Preliminaries}

Sets: We denote the closure and interior of a set $\mathcal{E}$ by $\bar{\mathcal{E}}$ and ${\mathcal{E}}^\circ$, respectively. The boundary of a set $\mathcal{E}$ defined as $\partial \mathcal{E} = \bar{\mathcal{E}} \setminus {\mathcal{E}}^\circ$. 

Linear Algebra: We denote by $\mathbb{R}^n$ the $n$-dimensional Euclidean space. For simplicity, we often use $(x_1, \ldots, x_n)$ to represent the column vector 
$\begin{bmatrix}
x_1 & x_2 & \ldots & x_n
\end{bmatrix}^T$. We use $e_i$ to denote the $i$-th basis vector (the vector of all zeros except for a one in the $i$-th position). We denote by $\Vert \cdot \Vert$ the Euclidean norm $\lVert x \rVert:=(\sum_i x_i^2)^{1/2}$ and by $\angle (u,v)$ the angle between vectors $u$ and $v$. Note that $\angle (u,v) = \arccos \big( \frac{u^T v}{\Vert u \Vert \Vert v \Vert}  \big)$. We use $\mathcal{B}_{r}(x_0) = \{ x \in \mathbb{R}^n: \Vert x-x_0 \Vert < r  \}$ and $\bar{\mathcal{B}}_{r}(x_0)$ to denote the open and closed ball, respectively, centered at $x_0$ of radius $r$.

Convex Sets and Functions: A set $\mathcal{C}$ in $\mathbb{R}^n$ is said to be convex if, for all $x$ and $y$ in $\mathcal{C}$ and all $t$ in the interval $(0, 1)$, the point $(1- t)x + ty$ also belongs to $\mathcal{C}$. A differentiable function $f$ is called strongly convex with parameter $\sigma > 0$ (or $\sigma$-strongly convex) if $(\nabla f(x) - \nabla f(y))^T (x-y) \geq \sigma \Vert x-y \Vert^2$ holds for all points $x, y$ in its domain. We denote the set of all $\sigma$-strongly convex functions by $\mathcal{S}(\sigma)$.

\section{Problem Statement}

We will consider two scenarios in this paper.  We first consider constraints on the gradients of the local functions at the location of the potential minimizer, and then consider constraints on the gradients inside a convex constraint set.

\subsection{Problem 1}
Consider two strongly convex functions $f_1: \mathbb{R}^n \rightarrow \mathbb{R}$ and $f_2: \mathbb{R}^n \rightarrow \mathbb{R}$.  The two functions $f_1$ and $f_2$ have strong convexity parameters $\sigma_1$ and $\sigma_2$, respectively, and minimizers $x_1^*$ and $x_2^*$, respectively.  Let $x$ denote the minimizer of $f_1 + f_2$, and suppose that the norm of the gradients of $f_1$ and $f_2$ must be bounded above by a finite number $L$ at $x$.  Our goal is to estimate the region $\mathcal{M}$ containing all possible values $x$ satisfying the above conditions.  More specifically, we wish to estimate the region 
\begin{multline}
\mathcal{M}(x_1^*,x_2^*,\sigma_1,\sigma_2,L) \triangleq \{x \in \mathbb{R}^n : \exists f_1 \in \mathcal{S}(\sigma_1),  \\
\exists f_2 \in \mathcal{S}(\sigma_2), \;\; \nabla f_1(x_1^*)  = 0, \;\; \nabla f_2(x_2^*)  = 0, \;\;  \\
\nabla f_1(x) = -\nabla f_2(x), \;\; \Vert \nabla f_1(x) \Vert = \Vert \nabla f_2(x) \Vert \leq L \}. \label{set M}
\end{multline}
For simplicity of notation, we will omit the argument of the set $\mathcal{M}(x_1^*,x_2^*,\sigma_1,\sigma_2,L)$ and write it as $\mathcal{M}$ or $\mathcal{M}(x_1^*, x_2^*)$.

\subsection{Problem 2}
Consider two strongly convex functions $f_1: \mathbb{R}^n \rightarrow \mathbb{R}$ and $f_2: \mathbb{R}^n \rightarrow \mathbb{R}$.  The two functions $f_1$ and $f_2$ have strong convexity parameters $\sigma_1$ and $\sigma_2$, respectively, and minimizers $x_1^*$ and $x_2^*$, respectively.  Suppose that we also have a compact convex set $\mathcal{C} \subset \mathbb{R}^n$ containing the minimizers $x_1^*$ and $x_2^*$.  Let $x$ denote the minimizer of $f_1 + f_2$ within the region $\mathcal{C}$.  The norm of the gradients of both functions $f_1$ and $f_2$ is bounded above by a finite number $L$ everywhere in the set $\mathcal{C}$.  Our goal is to estimate the region $\mathcal{N}$ containing all possible values $x_0 \in \mathcal{C}$ satisfying the above conditions.  More specifically, define $\mathcal{F}(\sigma, L, \mathcal{C})$ to be the family of functions that are $\sigma$-strongly convex and whose gradient norm is upper bounded by $L$ everywhere inside the convex set $\mathcal{C}$:
\begin{equation*}
\mathcal{F}(\sigma, L, \mathcal{C}) \triangleq \{ f : f \in \mathcal{S}(\sigma), \; \Vert \nabla f(x) \Vert \leq L, \;  \forall x \in \mathcal{C} \}.
\end{equation*}
Then, we wish to characterize the region
\begin{multline}
\mathcal{N}(x_1^*,x_2^*,\sigma_1,\sigma_2,L) \triangleq \{x \in \mathbb{R}^n : \exists f_1 \in \mathcal{F}(\sigma_1, L, \mathcal{C}), \\ \exists f_2 \in \mathcal{F}(\sigma_2, L, \mathcal{C}), \quad \nabla f_1(x_1^*)  = 0,  \\ 
\quad \nabla f_2(x_2^*)  = 0, \quad \nabla f_1(x) = -\nabla f_2(x)  \}. \label{set N}
\end{multline}
For simplicity of notation, we will omit the argument of the set $\mathcal{N}(x_1^*,x_2^*,\sigma_1,\sigma_2,L)$ and write it as $\mathcal{N}$ or $\mathcal{N}(x_1^*, x_2^*)$.

\subsection{A Preview of the Solution}
We provide two examples of the region containing the minimizer of the sum of $2$-dimensional functions in both scenarios in Fig. \ref{fig:convex}, where $x_1^*$ and $x_2^*$ are the minimizers of $f_1$ and $f_2$, respectively; we derive these regions in the rest of the paper. Notice that the region containing set $\mathcal{M}$ (the area  bounded by the red line) is bigger than the region containing set $\mathcal{N}$ (the area bounded by the blue line). In addition, even though we have changed the shape of convex set in the two examples, the minimizer regions are similar.      

\begin{figure}
\centering
\includegraphics[width=0.45\textwidth]{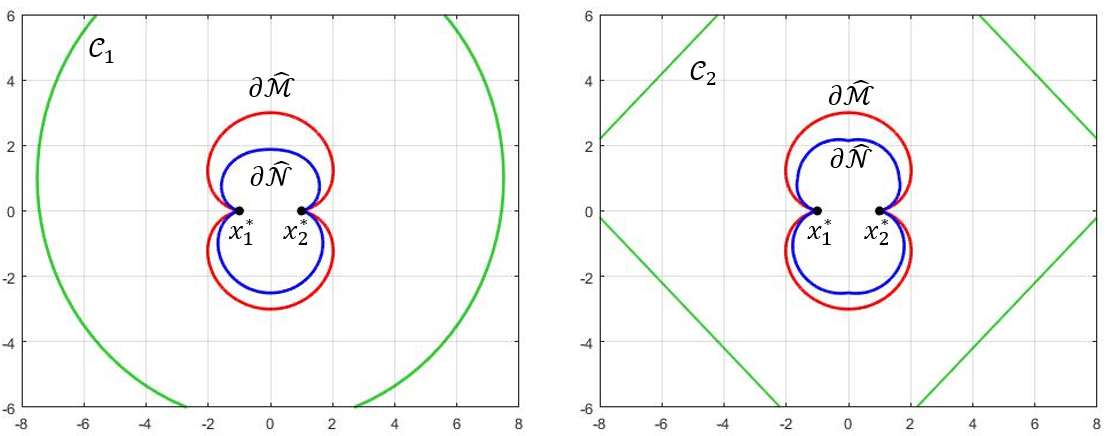}
\caption{\label{fig:convex}The red lines are the boundary of the region that contains $\mathcal{M}$, while the blue lines are the boundary of the region that contains $\mathcal{N}$, where convex sets $\mathcal{C}_1$ and $\mathcal{C}_2$ are a circle (Left) and a box (Right) respectively.}
\end{figure}

%%%%%%%%%%%%%%%%%%%%%%%%%%%%%%%%%%%%%%%%%%%%%%%%%%%%%%%%%%%%%%%%%%%%%%%%%%%%%%%%%%%%%%%%%%%%%%%%%%%%%%%%%%%%%%%%

\section{Problem 1: Gradient Constraint at Location of Potential Minimizer}

In this section, we consider the first scenario when the gradient constraint is imposed on the location of the potential minimizer and derive an approximation to the set $\mathcal{M}$ in \eqref{set M}.  %We begin by analyzing the necessary condition of the minimizer using a geometric approach and then state some properties related to the minimizer region and the region obtained from the necessary condition. Finally, we derive an upper bound on the region containing the possible minimizer in 2-dimensional and $(n+1)$-dimensional space.

Consider functions $f_1 \in \mathcal{S} (\sigma_1)$ with minimizer $x_1^*$ and $f_2 \in \mathcal{S} (\sigma_2)$ with minimizer $x_2^*$. Without loss of generality, we can assume $x_1^* = (-r,0, \ldots, 0) \in \mathbb{R}^{n}$ and $x_2^* = (r,0, \ldots, 0) \in \mathbb{R}^{n}$ for some $r \in \mathbb{R}_{> 0}$, since for any $x_1^*$ and $x_2^*$ such that $x_1^* \neq x_2^*$, we can find a unique affine transformation that maps the original minimizers into these values and also preserves the distance between these points i.e., $\Vert x_1^* - x_2^* \Vert = 2r$. The minimizer region in the original coordinates can then be obtained by applying the inverse transformation to the derived region.

%In the following Lemma, we only assume that $x_2^*$ align at the right-hand side of $x_1^*$ as shown in Fig. \ref{fig:angle1}.

We will be using the following functions throughout our analysis.  For $i \in \{1, 2\}$, define
\begin{equation}
 \tilde{\phi}_i(x, L) \triangleq \arccos \Big(\frac{\sigma_i}{L} \Vert x-x_i^* \Vert \Big),
 \label{eqn:phi_tilde}
 \end{equation}
 for all $x \in \mathbb{R}^{n}$ such that $\frac{\sigma_i}{L} \Vert x-x_i^* \Vert \le 1$. For simplicity of notation, if $L$ is a constant, we will omit the arguments and write it as $\tilde{\phi}_i(x)$ or $\tilde{\phi}_i$. Furthermore, for all $x \in \mathbb{R}^{n}$, define
 \begin{equation*}
 \psi(x) \triangleq \pi - \left(\alpha_2(x) - \alpha_1(x) \right), 
 \end{equation*}
 where $\alpha_i(x)$ is the angle between $x- x_i^*$ and $x_2^*-x_1^*$ i.e., $\alpha_i(x) \triangleq \angle ( x-x_i^*, x_2^* - x_1^* )$.

\begin{lemma}
  Necessary conditions for a point $x \in \mathbb{R}^n$ to be a minimizer of $f_1 + f_2$ when the gradients of $f_1$ and $f_2$ are bounded by $L$ at $x$ are (i) $\Vert x-x_i^* \Vert \leq \frac{L}{\sigma_i}$ for $i=1, 2$, and (ii) $\tilde{\phi}_1(x) + \tilde{\phi}_2(x) \geq \psi(x)$. \label{lem: angle}
\end{lemma} 

\iffalse
\begin{lemma}
Define $\tilde{\phi}_i(x) \triangleq \arccos (\frac{\sigma_i}{L} \Vert x-x_i^* \Vert)$ for $i=1, 2$, and $\psi(x) \triangleq \pi - \vert \alpha_2(x) - \alpha_1(x) \vert$ and $\alpha_i(x)$ is the angle between $x- x_i^*$ and $x_2^*-x_1^*$ given\footnote{SS: this is very hard to understand, because $x_0$ suddenly appears.  Do you need this condition on $\|x_0-x_i^*\|$ in order to define the two functions?  Can't the two functions always be defined, and we only need this condition in order to ensure that $x_0$ can be a minimizer?  If so, move this condition to the next sentence as a necessary condition for $x_0$ to be a minimizer.  Come back to fix this.} that $\Vert x_0-x_i^* \Vert \leq \frac{L}{\sigma_i}$ for $i=1, 2$. If $x_0$ is a minimizer of $f = f_1 + f_2$ and the gradients of $f_1$ and $f_2$ are bounded by $L$ at $x_0$, then $\tilde{\phi}_1(x_0) + \tilde{\phi}_2(x_0) \geq \psi(x_0)$.
\end{lemma} 
\fi

\begin{proof}
From the definition of strongly convex functions, 
\begin{equation*}
(\nabla f_i(x) - \nabla f_i(y))^T (x-y) \geq \sigma_i \Vert x-y \Vert^2
\end{equation*}
for all $x, y$ and for $i=1, 2$. Since $x_1^*$ and $x_2^*$ are the minimizers of $f_1$ and  $f_2$ respectively, we get 
\begin{align}
(\nabla f_i(x) - \nabla f_i(x_i^*))^T (x-x_i^*) &\geq \sigma_i \Vert x-x_i^* \Vert^2 \nonumber\\
\Rightarrow \quad \nabla f_i(x) ^T \frac{x-x_i^*}{\Vert x-x_i^* \Vert} &\geq \sigma_i \Vert x-x_i^* \Vert \geq 0. \label{strongly convex 1}
\end{align}
Let $u_i(x) \triangleq \frac{x-x_i^*}{\Vert x-x_i^* \Vert}$ be the unit vector in the direction of $x-x_i^*$ and $\phi_i(x) \triangleq \angle ( \nabla f_i(x), u_i(x) )$, with $0 \leq \phi_i(x) \leq \frac{\pi}{2}$ as shown in Fig. \ref{fig:angle1}. From \eqref{strongly convex 1}, we get
\begin{equation*}
\nabla f_i(x) ^T u_i(x) = \Vert \nabla f_i(x) \Vert \cos(\phi_i(x))\geq \sigma_i \Vert x-x_i^* \Vert. 
\end{equation*}
If $x$ is a candidate minimizer then we can apply the gradient norm constraint $\Vert \nabla f_i(x) \Vert \leq L$ to the above inequality to obtain
\begin{align}
%L \cos(\phi_i(x_0)) &\geq& \sigma_i \Vert x_0-x_i^* \Vert \nonumber\\
\cos(\phi_i(x)) &\geq \frac{ \sigma_i}{L} \Vert x-x_i^* \Vert. \label{angle eq 1}
\end{align}
If $\frac{ \sigma_i}{L} \Vert x-x_i^* \Vert \leq 1$ then $\phi_i(x) \leq \arccos (\frac{ \sigma_i}{L} \Vert x-x_i^* \Vert)$. On the other hand, if $\frac{ \sigma_i}{L} \Vert x-x_i^* \Vert > 1$ then there is no $\phi_i(x)$ that can satisfy the inequality \eqref{angle eq 1}. Therefore, if $\Vert x - x_1^* \Vert > \frac{L}{\sigma_1}$ or $\Vert x - x_2^* \Vert > \frac{L}{\sigma_2}$, we conclude that $x$ cannot be the minimizer of the function $f_1 + f_2$.

\begin{figure}
\centering
\includegraphics[width=0.25\textwidth]{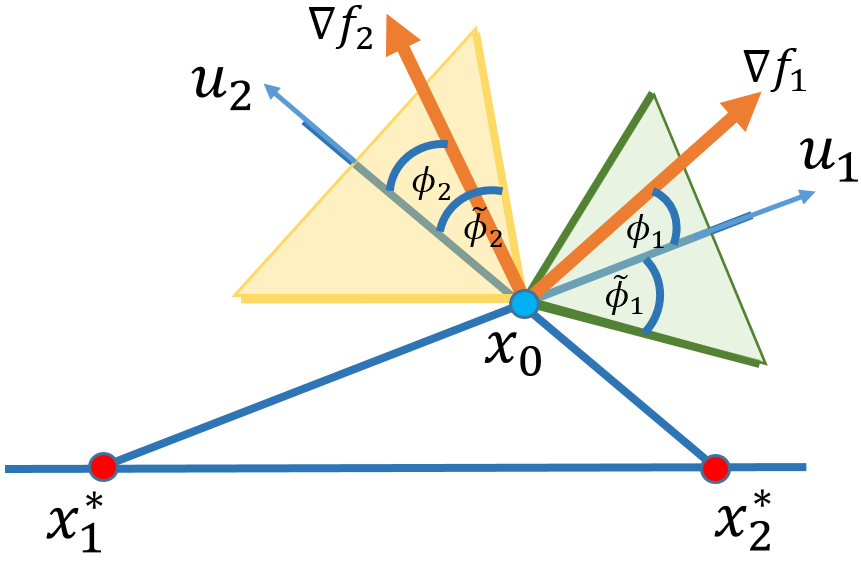}
\caption{The quantities $\phi_i(x_0)$ represent the angles between $\nabla f_i(x_0)$ and $u_i(x_0)$.  The quantities $\tilde{\phi}_i(x_0)$ represent the maximum possible values for $\phi_i(x_0)$ in order for $x_0$ to be a minimizer.  In other words, the angles $\phi_1(x_0)$ and $\phi_2(x_0)$ must lie in the shaded regions.}
\label{fig:angle1}
\end{figure}

Suppose that $\frac{\sigma_i}{L} \Vert x - x_i^* \Vert \leq 1$ for $i=1, 2$ so that $\arccos (\frac{ \sigma_1}{L} \Vert x-x_1^* \Vert)$ and $\arccos (\frac{ \sigma_2}{L} \Vert x - x_2^* \Vert)$ are well-defined. In order to capture the possible gradient of $f_1$ at point $x$, define a set of vectors whose norms are at most $L$ and satisfy \eqref{angle eq 1}:
\begin{multline*}
\mathcal{G}_1(x) \triangleq \Big\{g \in \mathbb{R}^{n} : \Vert g \Vert \leq L, \\
\angle (g,u_1(x)) \leq \arccos \Big(\frac{ \sigma_1}{L} \Vert x-x_1^* \Vert \Big) \Big\}.
\end{multline*}
Since $x$ can be the minimizer of the function $f_1 + f_2$ only when $\nabla f_1(x) =  -\nabla f_2(x)$, we define a set of vectors whose norms are at most $L$ and satisfy \eqref{angle eq 1} to capture the possible negated gradient vectors of $f_2$:   
\begin{multline*}
\mathcal{G}_2(x) \triangleq \Big\{g \in \mathbb{R}^{n} : \Vert g \Vert \leq L, \\
\angle (-g,u_2(x)) \leq \arccos \Big(\frac{ \sigma_2}{L} \Vert x-x_2^* \Vert \Big) \Big\}.
\end{multline*}
Note that $\phi_2(x)$ can be viewed geometrically as the angle between $-\nabla f_2(x)$ and $-u_2(x)$ as shown in Fig. \ref{fig:angle1}.  If $\mathcal{G}_1(x) \cap \mathcal{G}_2(x) = \emptyset$, then $x$ cannot be the minimizer of  the function $f_1 + f_2$ because it is not possible to choose $f_1$ and $f_2$ such that $\nabla f_i(x)$ satisfy inequality \eqref{angle eq 1} for $i=1, 2$ and $\nabla f_1(x) = -\nabla f_2(x)$ simultaneously.

\begin{figure}
\centering
\includegraphics[width=0.25\textwidth]{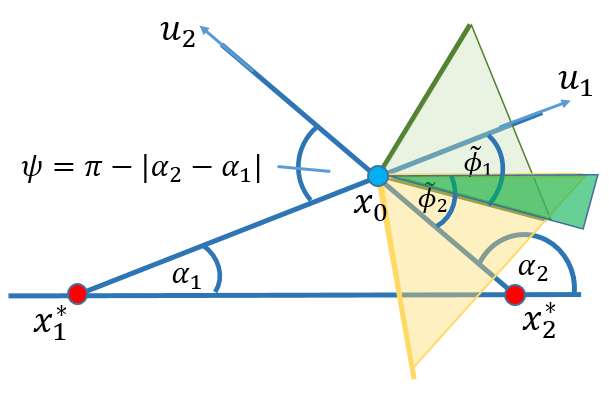}
\caption{The green region in the figure is the set $\mathcal{G}_1(x_0)$ and the yellow region is the set $\mathcal{G}_2(x_0)$.  These regions are defined by the angles $\tilde{\phi}_1$ and $\tilde{\phi}_2$. If these regions overlap, the point $x_0$ is a minimizer candidate.}
\label{fig:angle2}
\end{figure}

Recall that $\alpha_i(x) = \angle ( u_i(x), x_2^*-x_1^*)$ with $0 \leq \alpha_i(x) \leq \pi$ for $i=1, 2$, i.e.,
$
\alpha_i(x) = \arccos \Big( u_i(x)^T \frac{(x_2^*-x_1^*)}{\Vert x_2^*-x_1^* \Vert} \Big).
$
Note that $\alpha_2(x) \geq \alpha_1(x)$ due to the definition of $\alpha_i$. Then, the angle between $u_1(x)$ and $u_2(x)$ is $\alpha_2(x) - \alpha_1(x)$. Therefore, the angle between $u_1(x)$ and $-u_2(x)$ is equal to $\psi(x) = \pi - (\alpha_2(x) - \alpha_1(x) )$.

Let $\tilde{\phi}_i(x)$ be the maximum angle of $\phi_i(x)$ that satisfies inequality \eqref{angle eq 1}, i.e., as given by \eqref{eqn:phi_tilde}.
By the definition of $\tilde{\phi}_i(x)$, if $\tilde{\phi}_1(x) + \tilde{\phi}_2(x) \geq \psi(x)$, there is an overlapping region caused by $\tilde{\phi}_1(x)$ and $\tilde{\phi}_2(x)$ as shown in Fig. \ref{fig:angle2} and there exist gradients $\nabla f_1(x) \in \mathcal{G}_1(x)$ and $- \nabla f_2(x) \in \mathcal{G}_2(x)$ such that $\nabla f_1(x) = - \nabla f_2(x)$. On the other hand, if $\tilde{\phi}_1(x) + \tilde{\phi}_2(x) < \psi(x)$ then $\mathcal{G}_1(x) \cap \mathcal{G}_2(x) = \emptyset$  and it is not possible to choose gradients $\nabla f_1(x) \in \mathcal{G}_1(x)$ and $- \nabla f_2(x) \in \mathcal{G}_2(x)$ such that they cancel each other. In this case, we can conclude that this $x$ cannot be the minimizer of the function $f_1 + f_2$.
\end{proof}

Note that angles $\tilde{\phi}_1 (x)$, $\tilde{\phi}_2 (x)$, $\alpha_1(x)$, and $\alpha_2(x)$ can be expressed as a function of $\Vert x_1^* - x_2^* \Vert$, $\Vert x - x_1^* \Vert$, and $\Vert x - x_2^* \Vert$. 
Thus, from the proof of Lemma~\ref{lem: angle}, the inequality $\tilde{\phi}_1(x) + \tilde{\phi}_2(x) \geq \psi(x)$ depends only on the distance between the three points $x_1^*$, $x_2^*$, and $x$. Therefore, the candidate minimizer property of $x$ can be fully described by the 2-D picture in Fig. \ref{fig:angle2}.

Now we consider the relationship between set $\mathcal{M}$ in \eqref{set M} (which is the set that we want to identify) and certain other sets which we define below. Define the set
\begin{multline}
\hat{\mathcal{M}}(x_1^*,x_2^*) \triangleq \Big\{ x \in \mathbb{R}^n: \tilde{\phi}_1(x) + \tilde{\phi}_2(x) \geq \psi(x), \\
\Vert x-x_1^* \Vert \leq \frac{L}{\sigma_1}, \quad \Vert x-x_2^* \Vert \leq \frac{L}{\sigma_2}  \Big \}. \label{set: M_hat}
\end{multline}
Note that based on Lemma~\ref{lem: angle}, $\hat{\mathcal{M}}$ contains the minimizers of $f_1 + f_2$.

Define $\mathcal{H}$ to be the set of points such that there exist strongly convex functions (with given strong convexity parameters and minimizers) whose gradients can be bounded by $L$ at those points:
\begin{multline}
\mathcal{H}(x_1^*,x_2^*) \triangleq \{x \in \mathbb{R}^n : \exists f_1 \in \mathcal{S}(\sigma_1), \\
\exists f_2 \in \mathcal{S}(\sigma_2), \;\; \nabla f_1(x_1^*)  = 0, \;\; \nabla f_2(x_2^*)  = 0,  \\
\Vert \nabla f_1(x) \Vert \leq L, \;\; \Vert \nabla f_2(x) \Vert \leq L  \}. \label{set: H}
\end{multline}

Define $\mathcal{H}_i$ to be the set of points such that there exists a $\sigma_i$-strongly convex function $f_i$ with minimizer $x_i^*$ whose gradient is bounded by $L$ at those points: 
\begin{multline*}
\mathcal{H}_i(x_i^*) \triangleq \{x \in \mathbb{R}^n : \exists f_i \in \mathcal{S}(\sigma_i), \\ 
\nabla f_i(x_i^*)  = 0, \;  \Vert \nabla f_i(x) \Vert \leq L \}, \enspace i = 1, 2. 
\end{multline*}

\begin{lemma}
$\mathcal{M}(x_1^*,x_2^*) \subseteq \hat{\mathcal{M}}(x_1^*,x_2^*) \subseteq \mathcal{H}(x_1^*,x_2^*) $ and $\mathcal{H}(x_1^*,x_2^*) = \bar{\mathcal{B}}_{\frac{L}{\sigma_1}}(x_1^*) \cap \bar{\mathcal{B}}_{\frac{L}{\sigma_2}}(x_2^*)$.
\label{lem:nested_sets}
\end{lemma}

\begin{proof}
From Lemma~\ref{lem: angle}, we get $\mathcal{M}(x_1^*,x_2^*) \subseteq \hat{\mathcal{M}}(x_1^*,x_2^*)$. From the definition of a strongly convex function,
\begin{equation*}
(\nabla f_i(x) - \nabla f_i(y))^T (x-y) \geq \sigma_i \Vert x-y \Vert^2
\end{equation*}
for all $x, y$ where $i = 1, 2$. Substitute $x_i^*$ into $y$ to get
\begin{align}
(\nabla f_i(x) - \nabla f_i(x_i^*))^T (x-x_i^*) &\geq \sigma_i \Vert x-x_i^* \Vert^2 \nonumber\\
\Leftrightarrow \Vert \nabla f_i(x) \Vert \Vert x-x_i^* \Vert \cos (\phi_i(x)) &\geq \sigma_i \Vert x-x_i^* \Vert^2 \nonumber\\
\Rightarrow \quad L &\geq \sigma_i \Vert x-x_i^* \Vert \nonumber\\
\Leftrightarrow \quad \Vert x-x_i^* \Vert &\leq \frac{L}{\sigma_i} \label{max dist}
\end{align}
where the equality $\Vert \nabla f_i(x) \Vert \cos (\phi_i(x)) = L$ occurs when $\nabla f_i(x)$ is chosen such that $\Vert \nabla f_i(x) \Vert = L$ and $\nabla f_i(x)^T u_i(x) = L$. Note that the above sequence of inequalities uses the fact that $\Vert \nabla f_i(x) \Vert \leq L$ and $0 \leq \cos (\phi_i(x)) \leq 1$. Since $\bar{\mathcal{B}}_{r}(x_0) = \{ x: \Vert x-x_0 \Vert \leq r  \}$, from \eqref{max dist}, we have $\mathcal{H}_i(x_i^*) \subseteq \bar{\mathcal{B}}_{\frac{L}{\sigma_i}}(x_i^*)$. 

For the converse, consider $\hat{x} \in \bar{\mathcal{B}}_{\frac{L}{\sigma_i}}(x_i^*)$. By choosing a quadratic function $f_i(x) = \frac{1}{2} \hat{\sigma}_i (x - x_i^*)^T (x - x_i^*)$ where $\hat{\sigma}_i = \frac{L}{ \Vert \hat{x}-x_i^* \Vert }$, one can easily verify that $\hat{\sigma}_i \geq \sigma_i$ and $\Vert \nabla f_i(\hat{x}) \Vert = L$. So, we have $\mathcal{H}_i(x_i^*) \supseteq \bar{\mathcal{B}}_{\frac{L} {\sigma_i}}(x_i^*)$.

From the definition of $\mathcal{H}$ and $\mathcal{H}_i$, we get $\mathcal{H}(x_1^*,x_2^*) = \mathcal{H}_1(x_1^*) \cap \mathcal{H}_2(x_2^*) = \bar{\mathcal{B}}_{\frac{L}{\sigma_1}}(x_1^*) \cap \bar{\mathcal{B}}_{\frac{L}{\sigma_2}}(x_2^*)$.  Finally, since the conditions of the set $\mathcal{H}$ are the same as the last two conditions in the set $\hat{\mathcal{M}}$, we get $\hat{\mathcal{M}}(x_1^*,x_2^*) \subseteq \mathcal{H}(x_1^*,x_2^*)$.
\end{proof} 

The result from Lemma~\ref{lem:nested_sets} shows that the set $\hat{\mathcal{M}}$ contains the set $\mathcal{M}$ from \eqref{set M} within it.  Thus, we will derive the equation of the boundary of $\hat{\mathcal{M}}$ in $n$-dimensional space from the angles  $\tilde{\phi}_i$ defined in \eqref{eqn:phi_tilde}, and the necessary condition $\tilde{\phi}_1(x) + \tilde{\phi}_2 (x) \geq \psi(x)$. 

%################################ arXiv: Lemma 1D Analysis #################################
%\iffalse
From this point, we will denote $x = (z_1, \mathbf{z}) \in \mathbb{R}^n$ where $z_1 \in \mathbb{R}$ and $\mathbf{z} = (z_2, z_3, \ldots, z_n) \in \mathbb{R}^{n-1}$.

\begin{lemma}
(i) $x^*_1 \in \partial \hat{\mathcal{M}}$ if and only if $r \leq \frac{L}{2 \sigma_2}$. \\
(ii) $x^*_2 \in \partial \hat{\mathcal{M}}$ if and only if $r \leq \frac{L}{2 \sigma_1}$. 
\label{lem: 1D analyze}
\end{lemma}

\begin{proof}
Consider case (i) with $r \leq \frac{L}{2 \sigma_2}$. First, suppose $x = ( -r+\epsilon, \mathbf{0})$ where $0 < \epsilon < \min \{ \frac{L}{2 \sigma_1}, 2r\}$. Since $r \leq \frac{L}{2 \sigma_2}$, $x \in \mathcal{H}$. By the location of $x$, we get $\alpha_1(x) = 0$ and  $\alpha_2(x) = \pi$. Consequently, we obtain $\psi(x) = 0$. Since $0 \leq \tilde{\phi}_i \leq \frac{\pi}{2}$, the inequality $\tilde{\phi}_1(x) + \tilde{\phi}_2(x) \geq \psi(x)$ holds. This means that $x \in \hat{\mathcal{M}}$. 

Second, suppose $x = ( -r-\epsilon, \mathbf{0})$ where $0 < \epsilon < \min \{ \frac{L}{2 \sigma_1}, \frac{L}{\sigma_2} - 2r \}$ and $r < \frac{L}{2 \sigma_2}$. By the location of $x$, we get $x \in \mathcal{H}$, $\alpha_1 (x) = \pi$, and $\alpha_2 (x) = \pi$. Consequently, we obtain $\psi(x) = \pi$. In order to satisfy the inequality $\tilde{\phi}_1(x) + \tilde{\phi}_2(x) \geq \psi(x)$, we have to choose $\tilde{\phi}_1(x) = \frac{\pi}{2}$ and $\tilde{\phi}_2(x) = \frac{\pi}{2}$. However, since $\sigma_1 >0$, $L > 0$, and $\Vert x - x_1^* \Vert > 0$, we get $\tilde{\phi}_1(x) < \frac{\pi}{2}$ and conclude that $x \notin \hat{\mathcal{M}}$. Thus, we have $x_1^{\ast} \in \partial \hat{\mathcal{M}}$.

If $r > \frac{L}{2 \sigma_2}$, then $x_1^* \notin \mathcal{H}$ and therefore $x_1^* \notin \partial \hat{\mathcal{M}}$. Combining the analysis above, we can conclude that $x_1^* \in \partial \hat{\mathcal{M}}$ if and only if $r \leq \frac{L}{2 \sigma_2}$. A similar proof applies to case (ii).
\end{proof}
Define the set of points 
\begin{multline*}
\mathcal{T}_n (L) = \bigg\{ (z_1, \mathbf{z} )\in \mathbb{R}^n: 
\frac{z_1^2 + \Vert \mathbf{z} \Vert^2 -r^2}{d_1^2 d_2^2}  + \frac{\sigma_1 \sigma_2}{L^2} \\
= \sqrt{\frac{1}{d_1^2} - \frac{\sigma_1^2}{L^2}} \cdot \sqrt{\frac{1}{d_2^2} - \frac{\sigma_2^2}{L^2}}  \bigg\} 
\end{multline*}
where $d_1 = \sqrt{(z_1 + r)^2 + \Vert \mathbf{z} \Vert^2}$ and $d_2 = \sqrt{(z_1 - r)^2 + \Vert \mathbf{z} \Vert^2}$. For simplicity of notation, if $L$ is a constant, we will omit the argument and write it as $\mathcal{T}_n$. In addition, since $\mathcal{H}_i(x_i^*) = \bar{\mathcal{B}}_{\frac{L}{\sigma_i}}(x_i^*)$, we can write $\partial \mathcal{H}_i$ for $i \in \{1,2 \}$ as follows:
\begin{equation*}
\partial \mathcal{H}_1 = \bigg\{ (z_1, \mathbf{z} )\in \mathbb{R}^n: (z_1+r)^2 + \Vert \mathbf{z} \Vert^2 = \frac{L^2}{\sigma_1^2} \bigg\} 
\end{equation*}
\begin{equation*}
\partial \mathcal{H}_2 = \bigg\{ (z_1, \mathbf{z} )\in \mathbb{R}^n: (z_1-r)^2 + \Vert \mathbf{z} \Vert^2 = \frac{L^2}{\sigma_2^2} \bigg\}.  
\end{equation*}
%\fi
%################################ arXiv: Define T_n & R_n #################################

%################################ CDC: Define T_n  #################################
\iffalse
Define $x = (z_1, z_2, \ldots, z_n)$ and the set of points 
\begin{multline*}
\mathcal{T}_n (L) \triangleq \bigg\{ x \in \mathbb{R}^n: 
\frac{z_1^2 + \Vert \mathbf{z} \Vert^2 -r^2}{d_1^2 d_2^2}  + \frac{\sigma_1 \sigma_2}{L^2} \\
= \sqrt{\frac{1}{d_1^2} - \frac{\sigma_1^2}{L^2}} \cdot \sqrt{\frac{1}{d_2^2} - \frac{\sigma_2^2}{L^2}}  \bigg\} 
\end{multline*}
where $d_1(x) = \sqrt{(z_1 + r)^2 + \Vert \mathbf{z} \Vert^2}$, $d_2(x) = \sqrt{(z_1 - r)^2 + \Vert \mathbf{z} \Vert^2}$, $z_1 \in \mathbb{R}$ and $\mathbf{z} = (z_2, z_3, \ldots, z_n) \in \mathbb{R}^{n-1}$. For simplicity of notation, if $L$ is a constant, we will omit the argument and write it as $\mathcal{T}_n$.
\fi
%################################ CDC: Define T_n  #################################

\begin{lemma}
The set $\{ x \in \mathbb{R}^{n} : \tilde{\phi}_1(x) + \tilde{\phi}_2(x) = \pi - (\alpha_2(x) - \alpha_1(x)) \}$ is equivalent to $\mathcal{T}_n$.
\label{lem: T_n}
\end{lemma}

\begin{proof}
From Fig. \ref{fig:angle2}, for any point $x = (z_1, \mathbf{z}) \in \mathbb{R}^{n}$ with $x \notin \{x_1^{\ast}, x_2^{\ast}\}$, the $z_1$-axis equations are given by (with $x$ elided for notational convenience)
\begin{align}
z_1 = d_1 \cos \alpha_1 -r &= d_2 \cos \alpha_2 + r,  \nonumber\\
\Leftrightarrow \quad \cos \alpha_1 = \frac{z_1 + r}{d_1} 
\quad \text{and}& \quad
\cos \alpha_2 = \frac{z_1 - r}{d_2}. \label{cos alpha}
\end{align}
The $\mathbf{z}$-axes equations are given by
\begin{align}
\Vert \mathbf{z} \Vert &= d_1 \sin \alpha_1 = d_2 \sin \alpha_2, \nonumber\\
\Leftrightarrow \quad \sin \alpha_1 &= \frac{\Vert \mathbf{z} \Vert}{d_1}
\quad \text{and} \quad
\sin \alpha_2 = \frac{\Vert \mathbf{z} \Vert}{d_2}. \label{sin alpha}
\end{align}
Consider the equation
\begin{align}
\tilde{\phi}_1(x) + \tilde{\phi}_2(x) &= \pi - ( \alpha_2(x) - \alpha_1(x) ). \label{angle eq}
\end{align}
Since $0 \leq \tilde{\phi}_i \leq \frac{\pi}{2}$, we get $0 \leq \tilde{\phi}_1 + \tilde{\phi}_2 \leq \pi$. 
Since $0 \leq \alpha_i \leq \pi$ and $\alpha_2 \geq \alpha_1$,  $0 \leq \pi - (\alpha_2 - \alpha_1) \leq \pi$. 
Thus, since the cosine function is one-to-one for this range of angles, equation \eqref{angle eq} is equivalent to
\begin{align*}
	\cos(\tilde{\phi}_1 + \tilde{\phi}_2) &= \cos(\pi - (\alpha_2 - \alpha_1)) \\
	\Leftrightarrow \quad \cos(\tilde{\phi}_1 + \tilde{\phi}_2) &= - \cos (\alpha_2 - \alpha_1). 
\end{align*}
Expanding this equation and substituting \eqref{cos alpha}, \eqref{sin alpha}, and $\cos(\tilde{\phi}_i(x)) = \frac{ \sigma_i}{L} d_i$ for $i= 1,2$, we get
\begin{multline*}
	\frac{\sigma_1}{L} d_1 \cdot \frac{\sigma_2}{L} d_2 - \sqrt{1 - \Big(\frac{\sigma_1}{L} d_1 \Big)^2} \cdot \sqrt{1 - \Big(\frac{\sigma_2}{L} d_2 \Big)^2} \\
     = - \frac{z_1 -r}{d_2} \cdot\frac{z_1 +r}{d_1}  - \frac{\Vert \mathbf{z} \Vert}{d_2} \cdot \frac{\Vert \mathbf{z} \Vert}{d_1}. 
\end{multline*}
Dividing the above equation by $d_1 d_2$ and rearranging yields $\mathcal{T}_n$.\\
\end{proof}

%################################ CDC: Idea of Lemma 1D Analysis #################################
\iffalse
We next provide a lemma that will subsequently lead to the main result of this section, namely Theorem \ref{Thm1}. The proof of Lemma \ref{lem: 1D analyze} and Theorem \ref{Thm1} can be found in \cite{kuwaranancharoen2018location}.

\begin{lemma}
(i) $x^*_1 \in \partial \hat{\mathcal{M}}$ if and only if $r \leq \frac{L}{2 \sigma_2}$. \\
(ii) $x^*_2 \in \partial \hat{\mathcal{M}}$ if and only if $r \leq \frac{L}{2 \sigma_1}$. 
\label{lem: 1D analyze}
\end{lemma}
\fi
%################################ CDC: Idea of Lemma 1D Analysis #################################

%################################ arXiv: Lemma point on R_n #################################
%\iffalse
For convenience, we define $\gamma_i \triangleq \frac{L^2}{\sigma_i^2}$ for $i \in \{1, 2 \}$ and $\beta  \triangleq \frac{\sigma_2}{\sigma_1}$. We also define $\lambda_1  \triangleq \big(\frac{1 + \beta}{1 + 2 \beta} \big) \frac{\gamma_1}{2 r} - \frac{r}{1 + 2 \beta}$ and $\lambda_2  \triangleq -\big(\frac{1 + \beta}{2 +  \beta} \big) \frac{\gamma_2}{2 r} + \frac{\beta r}{2 + \beta}$.

In the following lemma, we will show that if we consider the points in $\partial \mathcal{H}_1$ or $\partial \mathcal{H}_2$, we can simplify the angle condition given in \eqref{set: M_hat} for $\hat{\mathcal{M}}$.
\begin{lemma}
Consider $x = (z_1, \mathbf{z}) \notin \{x_1^{\ast}, x_2^{\ast}\}$. \\
(i) If $x \in \partial \mathcal{H}_1$ then $\tilde{\phi}_1(x) + \tilde{\phi}_2(x) < \pi - (\alpha_2(x) - \alpha_1(x))$ if and only if $z_1 < \lambda_1$. \\
(ii) If $x \in \partial \mathcal{H}_2$ then $\tilde{\phi}_1(x) + \tilde{\phi}_2(x) < \pi - (\alpha_2(x) - \alpha_1(x))$ if and only if $z_1 > \lambda_2$. 
\label{lem: point on R_n}
\end{lemma}

\begin{proof}
Consider part (i). Since $x \in \partial \mathcal{H}_1$, we get $\|x-x_1^*\| = \frac{L}{\sigma_1}$ and thus $\tilde{\phi}_1(x) = 0$ from \eqref{eqn:phi_tilde}. Consider the inequality
\begin{align*}
\tilde{\phi}_1(x) + \tilde{\phi}_2(x) < \pi - (\alpha_2(x) - \alpha_1(x)).
\end{align*}
Substitute $\tilde{\phi}_1(x) = 0$ and take cosine of both sides of the inequality (and use \eqref{eqn:phi_tilde}) to get 
\begin{align*}
\frac{\sigma_2}{L}d_2 >  - \cos (\alpha_2(x) - \alpha_1(x)). 
\end{align*}
Expand the cosine and substitute the equations \eqref{cos alpha} and \eqref{sin alpha} to obtain
\begin{align}
\frac{\sigma_2}{L}d_2 >  -\frac{z_1^2 + \Vert \mathbf{z} \Vert^2 - r^2}{d_1 d_2}. \label{point in H ineq}
\end{align}
Since $x \in \partial \mathcal{H}_1$, we have $d_1 = \frac{L}{\sigma_1}$, and $\Vert \mathbf{z} \Vert^2 = \frac{L^2}{\sigma_1^2} - (z_1 + r)^2$. Also, $d_2^2 = (z_1 - r)^2 + \Vert \mathbf{z} \Vert^2 = (z_1 - r)^2 + \frac{L^2}{\sigma_1^2} - (z_1 + r)^2 = -4rz_1 + \frac{L^2}{\sigma_1^2}$. Multiply the inequality \eqref{point in H ineq} by $d_1 d_2$ and then substitute $d_1$, $\Vert \mathbf{z} \Vert^2$, and $d_2^2$ to get
\begin{align*}
&\frac{\sigma_2}{\sigma_1} \Big(-4r z_1 + \frac{L^2}{\sigma_1^2} \Big) > 2r^2 + 2r z_1 - \frac{L^2}{\sigma_1^2} \\
\Leftrightarrow \quad &z_1 \Big(2r + 4r \frac{\sigma_2}{\sigma_1} \Big) < \frac{\sigma_2}{\sigma_1} \cdot \frac{L^2}{\sigma_1^2} + \frac{L^2}{\sigma_1^2} - 2r^2 \\
\Leftrightarrow \quad &z_1 < \Big(\frac{1 + \beta}{1 + 2 \beta} \Big) \frac{\gamma_1}{2 r} - \frac{r}{1 + 2 \beta}.
\end{align*}
The proof of the second part is similar to the first part. \\
\end{proof}
In the following lemma, we consider the case when $r \leq \frac{L}{2 \sigma_1}$ and $r \leq \frac{L}{2 \sigma_2}$. The goals are to compare $\lambda_1$ with the maximum value of the $z_1$-component over all points of $\partial \mathcal{H}_1$ (which is $-r + \frac{L}{\sigma_1}$), and compare $\lambda_2$ with the minimum value of the $z_1$-component over all points of $\partial \mathcal{H}_2$ (which is $r - \frac{L}{\sigma_1}$), respectively.
\begin{lemma}
Consider $x = (z_1, \mathbf{z}) \in \mathbb{R}^n$. \\
(i) If $r \leq \frac{L}{2 \sigma_1}$ and $z_1 \leq \frac{L}{\sigma_1} -r$ then $z_1 \leq \lambda_1$, with equality only if $r = \frac{L}{2 \sigma_1}$ and $z_1 = \frac{L}{\sigma_1} -r$. \\
(ii) If $r \leq \frac{L}{2 \sigma_2}$ and $z_1 \geq r- \frac{L}{\sigma_2}$ then $z_1 \geq \lambda_2$, with equality only if $r = \frac{L}{2 \sigma_2}$ and $z_1 = r- \frac{L}{\sigma_2}$.
\label{lem: compare z_1 1}
\end{lemma}

\begin{proof}
Consider the first part of the lemma. First, we will rewrite each inequality. The inequality $r \leq \frac{L}{2 \sigma_1}$ becomes $\frac{L}{\sigma_1 r} \geq 2$, $z_1 \leq \frac{L}{\sigma_1} -r$ becomes $\frac{z_1}{r} \leq \frac{L}{\sigma_1 r}- 1$, and $z_1 \leq \big(\frac{1 + \beta}{1 + 2 \beta} \big) \frac{\gamma_1}{2 r} - \frac{r}{1 + 2 \beta}$ becomes $\frac{z_1}{r} \leq \big(\frac{1 + \beta}{1 + 2 \beta} \big) \frac{\gamma_1}{r^2} - \frac{1}{1 + 2 \beta}$. For simplicity, we define a new variable $\chi_1 = \frac{L}{\sigma_1 r}$. Note that $\frac{\gamma_1}{r^2} = \chi_1^2$. Then we need to show the equivalent statement of the lemma that if $\chi_1 \geq 2$ then $\chi_1 - 1 \leq \frac{1}{2} \big(\frac{1 + \beta}{1 + 2 \beta} \big) \chi_1^2 - \frac{1}{1 + 2 \beta}$, with equality only if $\chi_1 = 2$. 

Consider
\begin{align*}
\chi_1 - 1 \leq \frac{1}{2} \Big(\frac{1 + \beta}{1 + 2 \beta} \Big) \chi_1^2 - \frac{1}{1 + 2 \beta}.
\end{align*}
Multiplying both sides by $1+ 2 \beta$, we get 
\begin{align*}
&\chi_1 + 2 \beta  \chi_1 - 2 \beta -1 \leq \frac{1}{2} (1 + \beta) \chi_1^2 -1 \\
\Leftrightarrow \quad &\frac{1}{2} (1 + \beta) \chi_1^2 - (1 + 2 \beta) \chi_1 + 2 \beta \geq 0 \\
\Leftrightarrow \quad &(\chi_1 - 2) \Big( \frac{1}{2} (1 + \beta) \chi_1 - \beta \Big) \geq 0 \\
\Leftrightarrow \quad &\chi_1 \leq \frac{2 \beta}{1 + \beta} \quad \text{or} \quad \chi_1 \geq 2.
\end{align*}
We conclude that $\chi_1 \geq 2$ implies $\chi_1 - 1 \leq \frac{1}{2} \Big(\frac{1 + \beta}{1 + 2 \beta} \Big) \chi_1^2 - \frac{1}{1 + 2 \beta}$ with equality only if $\chi_1 = 2$. The proof of the second part of this lemma is similar to the first part. \\
\end{proof}
\begin{theorem}
If $r \leq \frac{L}{2} \cdot \min \big\{ \frac{1}{\sigma_1}, \frac{1}{\sigma_2} \big\}$ then the boundary $\partial \hat{\mathcal{M}}$ is given by $\mathcal{T}_n \cup \{ x^*_1, x^*_2 \}$. \label{Thm1}
\end{theorem}

\begin{proof}
Assume without loss of generality that $\sigma_1 \geq \sigma_2$. We want to show that if $x \in \partial \mathcal{H}$ then $\tilde{\phi}_1(x) + \tilde{\phi}_2(x) < \pi - (\alpha_2(x) - \alpha_1(x))$ except for $x \in \{ x_1^*, x_2^* \}$.

Suppose $x \in \partial \mathcal{H}_1$. Since $\mathcal{H}_1$ is closed and defined on $z_1 \in \big[ -r- \frac{L}{\sigma_1}, -r+ \frac{L}{\sigma_1} \big]$, from Lemma \ref{lem: compare z_1 1}, we get $z_1 \leq \lambda_1$ with equality only if $r = \frac{L}{2 \sigma_1}$ and $z_1 = \frac{L}{\sigma_1} - r$. If $z_1 < \lambda_1$, from Lemma \ref{lem: point on R_n}, we obtain $\tilde{\phi}_1(x) + \tilde{\phi}_2(x) < \psi(x)$. On the other hand, if $z_1 = \lambda_1$, from Lemma \ref{lem: compare z_1 1}, we get $r = \frac{L}{2 \sigma_1}$ and $z_1 = \frac{L}{\sigma_1} - r$. This means that $r = \frac{L}{2 \sigma_1}$ and $z_1 = \frac{L}{2 \sigma_1}$. Since $x \in \partial \mathcal{H}_1$ and $z_1 = r$, we conclude that $x = x_2^*$.

From the assumption $\sigma_1 \geq \sigma_2$ and the inequality $r \leq \frac{L}{2 \sigma_1}$, we get $r \leq \frac{L}{2 \sigma_2}$. We can similarly show that if $x \in \partial \mathcal{H}_2$ then $\tilde{\phi}_1(x) + \tilde{\phi}_2(x) < \psi(x)$ except for $x = x_1^*$ by using part (ii) of the Lemma  \ref{lem: point on R_n} and \ref{lem: compare z_1 1}.

Recall the definition of $\hat{\mathcal{M}}$ and Lemma~\ref{lem:nested_sets}. The boundary $\partial \hat{\mathcal{M}} \setminus \{x_1^*, x_2^*\}$ can be classified into 2 disjoint types. The first type consists of points $x$ with the following property:
$\tilde{\phi}_1(x) + \tilde{\phi}_2(x) = \pi - ( \alpha_2(x) - \alpha_1(x) )$ 
for which an example is shown in Fig. \ref{fig:angle3}.
\begin{figure}
\centering
\includegraphics[width=0.25\textwidth]{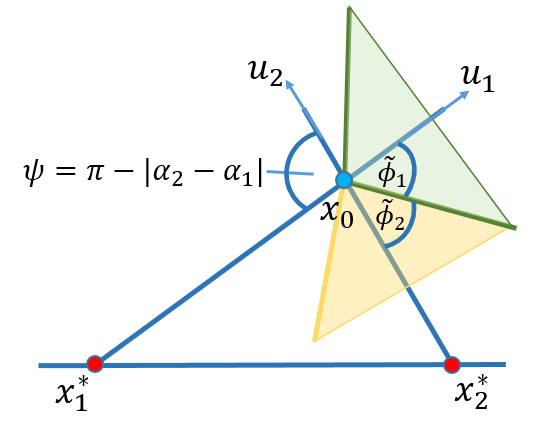}
\caption{\label{fig:angle3} The sets of gradients at a point on the boundary $\partial \hat{\mathcal{M}}$ that is not on the boundary $\partial \mathcal{H}$.  In this case, $\tilde{\phi}_1(x_0) + \tilde{\phi}_2(x_0) = \psi(x_0)$.}
\end{figure}
The second type consists of points $x$ with the following property:
$\tilde{\phi}_1(x) + \tilde{\phi}_2(x) > \pi - ( \alpha_2(x) - \alpha_1(x) )$
for which an example is shown in Fig. \ref{fig:angle4}. Note that $\alpha_1(x)$ and $\alpha_2(x)$ are not defined if $x \in \{ x^*_1,x^*_2 \}$. 

\begin{figure}
\centering
\includegraphics[width=0.30\textwidth]{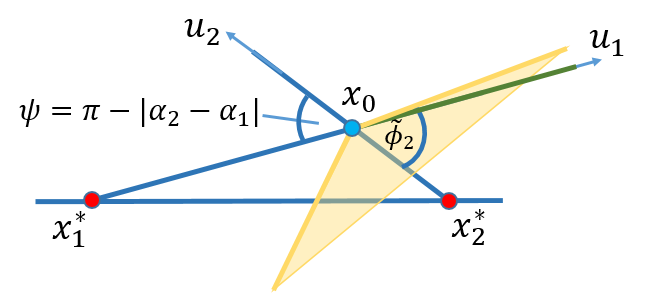}
\caption{\label{fig:angle4} The sets of gradients at a point on the boundary $\partial \hat{\mathcal{M}}$ that is also on the boundary $\partial \mathcal{H}$.  In this case, $\cos(\tilde{\phi}_1(x_0)) = 1$ and $\cos(\tilde{\phi}_2(x_0)) = \frac{ \sigma_2}{L} \Vert x_0-x_2^* \Vert $; however, $\tilde{\phi}_1(x_0) + \tilde{\phi}_2(x_0) > \psi(x_0)$.}
\end{figure} 

Consider the second type. We can separate it into three different cases as follows (recall that $\cos(\tilde{\phi}_i(x)) = \frac{\sigma_i}{L}\|x-x_i^*\|)$: \\
(i) $\cos(\tilde{\phi}_1(x)) < 1$ and $\cos(\tilde{\phi}_2(x)) < 1$. \\
(ii) $\cos(\tilde{\phi}_1(x)) = 1$ or $\cos(\tilde{\phi}_2(x)) = 1$.

We will argue that the point $x$ that satisfies the first case cannot be in $\partial \hat{\mathcal{M}}$. Since $\cos(\tilde{\phi}_1(x)) < 1$ and $\cos(\tilde{\phi}_2(x)) < 1$, we know that $x \in \mathcal{H}^\circ$. Since $\tilde{\phi}_1(x) + \tilde{\phi}_2(x) > \psi(x)$, there exists $\epsilon > 0$ such that for all $x_0 \in \mathcal{B}_{\epsilon} (x) \subset \mathcal{H}^\circ$, we have $\tilde{\phi}_1(x_0) + \tilde{\phi}_2(x_0) > \psi(x_0)$. So, $x \notin \partial \hat{\mathcal{M}}$.

Next, consider the point $x$ that satisfies the second case. From the definition of $\tilde{\phi}_i$ in \eqref{eqn:phi_tilde}, we get $\partial \mathcal{H}_i = \{ x: \cos(\tilde{\phi}_i(x)) = 1 \}$ for $i \in \{ 1,2 \}$. So, $x \in \partial \mathcal{H}_1 \cup \partial \mathcal{H}_2$. However, as discussed above, this makes $\tilde{\phi}_1(x) + \tilde{\phi}_2(x) < \psi(x)$ except for $x \in \{ x^*_1,x^*_2 \}$. Therefore, the point $x$ in the set $\{ x: \tilde{\phi}_1(x) + \tilde{\phi}_2(x) > \psi(x) \}$ cannot be in the boundary $\partial \hat{\mathcal{M}}$ and so $x \in \partial \hat{\mathcal{M}} \setminus \{x_1^*, x_2^* \}$ must satisfy $\{ x: \tilde{\phi}_1(x) + \tilde{\phi}_2(x) = \psi(x) \}$. From Lemma \ref{lem: T_n}, the set $\{ x: \tilde{\phi}_1(x) + \tilde{\phi}_2(x) = \psi(x) \}$ is equivalent to $\mathcal{T}_n$ and from Lemma \ref{lem: 1D analyze}, $\{ x^*_1,x^*_2 \} \in \partial \hat{\mathcal{M}} $. We conclude that if $r \leq \frac{L}{2 \sigma_1}$ then $\partial \hat{\mathcal{M}} = \mathcal{T}_n \cup \{ x^*_1, x^*_2 \}$.
\end{proof}
%\fi
%################################ ArXiv: Main Thm 1 case ################################

An example of the region $\hat{\mathcal{M}}$ given by Theorem \ref{Thm1} is shown in Fig. \ref{fig:f5}.

\section{Problem 2: Gradient Constraint on Convex Set}

%################################### CDC: Intro Convex #######################################
\iffalse
In this section, we consider the second scenario when the gradient constraint is imposed on a given convex set in which the minimizers of two original functions are embedded. We begin by analyzing the necessary condition for any given point to be a minimizer using a geometric approach and then state the relationship among certain sets related to the minimizer region. Finally, the equation of a region of possible minimizers in $n$-dimensional space is presented. For Lemma \ref{lem: grad_convex} and \ref{lem: grads_convex}, and Theorem \ref{Thm2}, we  will discuss the main ideas of the proof briefly. The complete proof is provided at \cite{kuwaranancharoen2018location}. 
\fi
%################################### CDC: Intro Convex #######################################

%################################### ArXiv: Intro Convex #######################################
%\iffalse
In this section, we consider the second scenario when the gradient constraint is imposed on a given convex set in which the minimizers of two original functions are embedded. We begin by analyzing the necessary condition for any given point to be a minimizer using a geometric approach and then state the relationship among certain sets related to the minimizer region. Finally, the equation of a region of possible minimizers in $n$-dimensional space is presented. 
%\fi
%################################### ArXiv: Intro Convex #######################################

Let $d(x_0, \partial \mathcal{C})$ be the infimum distance between $x_0$ and the boundary of a convex set $\mathcal{C}$, i.e.,
\begin{equation*}
d(x_0, \partial \mathcal{C}) \triangleq \inf_{x \in \partial \mathcal{C}} \Vert x - x_0 \Vert.
\end{equation*}

\begin{lemma}
Suppose $\mathcal{C}$ is a compact convex set and $x_0$ is a point in $\mathcal{C}$. Suppose $f \in \mathcal{S} (\sigma)$, and the norm of the gradient of $f$ in $\mathcal{C}$ is bounded by $L$, i.e., $\Vert \nabla f(x) \Vert \leq L$, $\forall x \in \mathcal{C}$. Then 
\begin{equation*}
\Vert \nabla f(x_0) \Vert \leq L - \sigma d(x_0, \partial \mathcal{C}) 
\end{equation*} \label{lem: grad_convex}
\end{lemma}

%########################## CDC: Idea of Gradient with Convex ##############################
\iffalse
The main idea of the proof of this lemma is that the norm of the gradient at $x_0$ (i.e., $\Vert \nabla f(x_0) \Vert$) plus the additional gradient increase from $x_0$ to the boundary $\partial \mathcal{C}$ must not exceed $L$. However, the distance from $x_0$ to the boundary is bounded below by $d(x_0, \partial C)$, so by rearranging the inequality, we obtain the result.
\fi
%########################## CDC: Idea of Gradient with Convex ##############################

%########################## arXiv: Proof of Gradient with Convex ##############################
%\iffalse
\begin{proof}
From the definition of strongly convex functions, 
\begin{equation*}
(\nabla f(x) - \nabla f(y))^T (x-y) \geq \sigma \Vert x-y \Vert^2.
\end{equation*}
Let $x_0$ be a point in the convex set $\mathcal{C}$. For any point $x$ on the boundary of the convex set, we have 
\begin{align}
(\nabla f(x) - \nabla f(x_0))^T (x-x_0) \geq \sigma \Vert x-x_0 \Vert^2 \nonumber\\
\Leftrightarrow \; \nabla f(x)^T (x-x_0) \geq \nabla f(x_0)^T (x-x_0) + \sigma  \Vert x-x_0 \Vert^2 \nonumber\\
\Rightarrow \; \Vert \nabla f(x) \Vert \Vert x-x_0 \Vert \geq \nabla f(x_0)^T (x-x_0) + \sigma  \Vert x-x_0 \Vert^2. \nonumber
\end{align}
Since $\Vert \nabla f(x) \Vert \leq L$ for all $x \in \partial \mathcal{C}$,
\begin{align}
\nabla f(x_0)^T \frac{x-x_0}{\Vert x-x_0 \Vert} + \sigma \Vert x-x_0 \Vert \leq L. \nonumber
\end{align}
Let $\theta$ be the angle between $\nabla f(x_0)$ and a unit vector in the direction of $x-x_0$. The above inequality becomes
\begin{equation*}
\Vert \nabla f(x_0) \Vert \cos \theta \leq L - \sigma \Vert x-x_0 \Vert. 
\end{equation*}
We can always choose $x \in \partial \mathcal{C}$ so that $x-x_0$ is collinear with $\nabla f(x_0)$. By this choice of $x$, we get 
\begin{equation*}
\Vert \nabla f(x_0) \Vert \leq L - \sigma \Vert x-x_0 \Vert. 
\end{equation*}
Since $\Vert x-x_0 \Vert \geq d(x_0, \partial \mathcal{C})$ for all $x \in \partial \mathcal{C}$, we obtain
\begin{align*}
\Vert \nabla f(x_0) \Vert \leq L - \sigma d(x_0, \partial \mathcal{C}). 
\end{align*}
\end{proof}
%\fi
%########################## arXiv: Proof of Gradient with Convex ##############################

\begin{lemma}
Suppose $\mathcal{C}$ is a compact convex set. Let $f_1 \in \mathcal{S} (\sigma_1)$, $f_2 \in \mathcal{S} (\sigma_2)$, $x_0$ be the minimizer of $f_1+f_2$ over the set $\mathcal{C}$ and $\hat{L}$ be the norm of the gradient of $f_1$ and $f_2$ at $x_0$. If the norm of the gradient of $f_1$ and $f_2$ in $\mathcal{C}$ is bounded by $L$, i.e., $\Vert \nabla f_i(x) \Vert \leq L$, $\forall x \in \mathcal{C}$, $i=1, 2$ then \label{lem: grads_convex}
\begin{align*}
\hat{L} \leq L - \min(\sigma_1, \sigma_2) \times d(x_0, \partial \mathcal{C}).
\end{align*} 
\end{lemma}

%########################## CDC: Idea of Gradient w Convex 2 func ##############################
\iffalse
In order to prove this lemma, we use the result from Lemma \ref{lem: grad_convex} and apply it to the functions $f_1$ and $f_2$. Since the gradients at $x_0$ are equal i.e., $\Vert \nabla f_1(x_0) \Vert = \Vert \nabla f_2(x_0) \Vert$, the minimum growth rate of gradient from $x_0$ to $x$ is determined by $\min(\sigma_1, \sigma_2)$.   
\fi
%########################## CDC: Idea of Gradient w Convex 2 func ##############################

%########################## arXiv: Proof of Gradient w Convex 2 func ##############################
%\iffalse
\begin{proof}
Consider strongly convex functions $f_1 \in \mathcal{S}(\sigma_1)$ and $f_2(x) \in \mathcal{S}(\sigma_2)$. If $x_0$ is in $\mathcal{C}$, from Lemma~\ref{lem: grad_convex}, we get
\begin{align*}
\Vert \nabla f_i(x_0) \Vert \leq L - \sigma_i  d(x_0, \partial \mathcal{C}) \quad \text{for} \quad i=1,2 
\end{align*}
Since $x_0$ is the minimizer of the sum of two strongly convex functions, it must satisfy $\Vert \nabla f_1(x_0) \Vert = \Vert \nabla f_2(x_0) \Vert$. Thus,
\begin{equation}
\Vert \nabla f_1(x_0) \Vert = \Vert \nabla f_2(x_0) \Vert \leq L - \min(\sigma_1, \sigma_2)  \Vert x-x_0 \Vert. \nonumber
\end{equation}
Since $\Vert \nabla f_1(x_0) \Vert = \Vert \nabla f_2(x_0) \Vert = \hat{L}$, the result follows. \\
\end{proof}
%\fi
%########################## arXiv: Proof of Gradient w Convex 2 func ##############################

\begin{figure}
\centering
\includegraphics[width=0.225\textwidth]{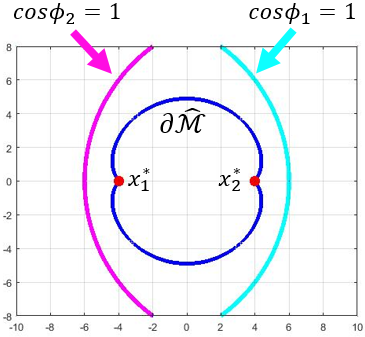}
\caption{\label{fig:f5}The boundary $\partial \hat{\mathcal{M}}$ (blue line) is plotted given original minimizers $x_1^*=(-4,0)$ and $x_2^*=(4,0)$ and parameters $\sigma_1 = \sigma_2 = 1$ and $L = 10$.}
\end{figure}

As before, without loss of generality, we can assume $x_1^* = (-r,0, \ldots, 0) \in \mathbb{R}^{n}$ and $x_2^* = (r,0, \ldots, 0) \in \mathbb{R}^{n}$ since for any minimizers $x_1^*$ and $x_2^*$, and a convex set $\mathcal{C}$, we can find a unique affine transformation that maps the original minimizers into $(-r,0, \ldots, 0)$ and $(r,0, \ldots, 0)$ respectively and also preserves the distance between these points, i.e., $\Vert x_1^* - x_2^* \Vert = 2r$. This transformation also uniquely maps the original convex set $\mathcal{C}$ into a new convex set $\mathcal{C}'$.

With the above assumption, we can now modify Lemma~\ref{lem: angle} with the new bound $\hat{L}$ on $\Vert \nabla f_i(x_0) \Vert$, provided by Lemma~\ref{lem: grads_convex}. Define a function
\begin{align*}
\tilde{L} (x) \triangleq L - \min(\sigma_1, \sigma_2) \times d(x, \partial \mathcal{C}) \quad \text{for} \quad x \in \mathcal{C}.
\end{align*}
\begin{lemma}
	Necessary conditions for a point $x \in \mathbb{R}^n$ to be a minimizer of $f_1 + f_2$ when the gradients of $f_1$ and $f_2$ are bounded by $L$ in a convex set $\mathcal{C}$ are (i) $\Vert x-x_i^* \Vert \leq \frac{1}{\sigma_i} \tilde{L}(x)$ for $i=1, 2$, and (ii) $\tilde{\phi}_1(x, \tilde{L}) + \tilde{\phi}_2(x, \tilde{L}) \geq \psi(x)$. \label{lem: angle2}
\end{lemma}

The proof is the same as Lemma~\ref{lem: angle} except that we use $\Vert \nabla f_i(x) \Vert \leq \tilde{L} (x)$ instead of $\Vert \nabla f_i(x) \Vert \leq L$. \\

Now we consider the relationship between the set $\mathcal{N}$ in \eqref{set N} (which is the set that we want to identify) and other sets which we will define below. 
Recall the definition of $\mathcal{N}$ from (\ref{set N}) where $\mathcal{F}(\sigma, L, \mathcal{C}) = \{ f : f \in \mathcal{S}(\sigma), \quad \Vert \nabla f(x) \Vert \leq L, \quad  \forall x \in \mathcal{C} \}$ for a given convex set $\mathcal{C}$.

We define $\hat{\mathcal{N}}$ as
\begin{multline*}
\hat{\mathcal{N}}(x_1^*,x_2^*) \triangleq \Big\{ x \in \mathbb{R}^n: \tilde{\phi}_1(x, \tilde{L}) + \tilde{\phi}_2(x, \tilde{L}) \geq \psi(x), \\
\Vert x-x_1^* \Vert \leq \frac{1}{\sigma_1}  \tilde{L}(x), \quad \Vert x-x_2^* \Vert \leq \frac{1}{\sigma_2}  \tilde{L}(x)  \Big \} 
\end{multline*}
where $\tilde{L} (x) = L - \min(\sigma_1, \sigma_2) \times d(x, \partial \mathcal{C})$. Note that unlike $L$, $\tilde{L} (x)$ is a function of $x$. By Lemma~\ref{lem: angle2}, $\hat{\mathcal{N}}$ contains the minimizers of $f_1 + f_2$ and $\mathcal{N}(x_1^*,x_2^*) \subseteq \hat{\mathcal{N}}(x_1^*,x_2^*)$.

Define $\mathcal{I}$ to be the set 
\begin{multline*}
\mathcal{I}(x_1^*,x_2^*) \triangleq \{x \in \mathbb{R}^n : \exists f_1 \in \mathcal{S}(\sigma_1), \quad \exists f_2 \in \mathcal{S}(\sigma_2),  \\
\nabla f_1(x_1^*)  = 0, \quad \nabla f_2(x_2^*)  = 0,  \\
\Vert \nabla f_1(x) \Vert \leq \tilde{L} (x), \quad \Vert \nabla f_2(x) \Vert \leq \tilde{L} (x)  \}. 
\end{multline*}

Define $\mathcal{I}_i$, $i=1, 2$, to be the set of points such that there exists a strongly convex function $f_i$ whose minimizer is $x_i^*$ and whose gradient can be bounded by $\tilde{L}$ at $x$:
\begin{multline*} 
\mathcal{I}_i(x_i^*) \triangleq \{x \in \mathbb{R}^n : \exists f_i \in \mathcal{S}(\sigma_i), \; \nabla f_i(x_i^*)  = 0,  \\ 
\Vert \nabla f_i(x) \Vert \leq \tilde{L} (x) \}.
\end{multline*}

\begin{lemma}
$\mathcal{N}(x_1^*,x_2^*) \subseteq \hat{\mathcal{N}}(x_1^*,x_2^*) \subseteq \mathcal{I}(x_1^*,x_2^*) $, $\mathcal{I}(x_1^*,x_2^*) = \mathcal{I}_1(x_1^*) \cap \mathcal{I}_2(x_2^*)$, and $\hat{\mathcal{N}}(x_1^*,x_2^*) \subseteq \hat{\mathcal{M}}(x_1^*,x_2^*)$ for all $x \in \mathcal{C}$. \label{lem:nested_sets2}
\end{lemma} 
\begin{proof}
The first and second parts are similar to the proof of Lemma~\ref{lem:nested_sets}. However, we cannot simplify the set $\mathcal{I}_i$ further (unlike the set $\mathcal{H}_i$ in Lemma~\ref{lem:nested_sets}) since $\mathcal{I}_i$ depends on the convex set $\mathcal{C}$ (via $\tilde{L}$).

Since the gradient $\tilde{L} (x)$ is no greater than $L$ for all $x \in \mathcal{C}$, the third part $\hat{\mathcal{N}}(x_1^*,x_2^*) \subseteq \hat{\mathcal{M}}(x_1^*,x_2^*)$ follows.    
\end{proof}

We can interpret Lemma~\ref{lem:nested_sets2} as follows. The constraints $\exists f_i \in \mathcal{F}(\sigma_i, L, \mathcal{C})$ for $i=1, 2$ in the set $\mathcal{N}$ are shifted to a looser constraint on their gradients, i.e., $\| \nabla f_i(x)\| \le L$ for all $x\in C$ becomes $\| \nabla f_i(x)\| \le \tilde{L}(x)$, where $\tilde{L} (x) = L - \min(\sigma_1, \sigma_2) \times d(x, \partial \mathcal{C})$. This simplifies the analysis significantly but potentially introduces conservatism. 

\begin{theorem}
If $\hat{\mathcal{M}}(x_1^*,x_2^*) \subseteq \mathcal{I}^{\circ}(x_1^*,x_2^*)$ and $r \leq \frac{L}{2} \times \min \big\{ \frac{1}{\sigma_1}, \frac{1}{\sigma_2} \big\}$, then $\partial \hat{\mathcal{N}}$ is given by $\mathcal{T}_n (\tilde{L}) \cup \{x_1^*, x_2^* \}$. \label{Thm2}
\end{theorem}

%############################# CDC: Idea of Theorem w Convex ##############################
\iffalse
The proof of Theorem \ref{Thm2} in \cite{kuwaranancharoen2018location} is obtained by noting from Theorem \ref{Thm1} and Lemma \ref{lem:nested_sets2} that $\hat{\mathcal{N}} \subset \mathcal{I}$. Then, as in the proof of Theorem \ref{Thm1}, the boundary $\partial \hat{\mathcal{N}}$ is shown to be described only by $\mathcal{T}_n (\tilde{L})$.
\fi
%############################# CDC: Idea of Theorem w Convex ##############################

%############################# arXiv: Proof of Theorem w Convex ##############################
%\iffalse
\begin{proof}
Consider a point $x = (z_1, \mathbf{0})$ where $z_1 \in  (-r, r)$ i.e., a point in between $x_1^*$ and $x_2^*$. Then, $\alpha_1(x) = 0$ and $\alpha_2(x) = \pi$, so we get $\psi(x) = 0$. Since $x \in \hat{\mathcal{M}}(x_1^*,x_2^*) \subseteq \mathcal{I}^{\circ}(x_1^*,x_2^*)$, we have 
\begin{align*}
\cos(\phi_i(x)) &\geq \frac{ \sigma_i}{ \tilde{L} } \Vert x-x_i^* \Vert \quad \text{and} \quad \frac{ \sigma_i}{ \tilde{L} } \Vert x-x_i^* \Vert \leq 1. 
\end{align*}
We get $\tilde{\phi}_i (x,\tilde{L}) \geq 0$ for $i = \{1,2 \}$ so the angle inequality $\tilde{\phi}_1(x, \tilde{L}) + \tilde{\phi}_2(x, \tilde{L}) \geq \psi(x)$ holds. On the other hand, $x$ cannot be the new minimizer when $z_1 \in (-\infty, -r) \cup (r, \infty)$ by using similar argument in the proof of Lemma \ref{lem: 1D analyze}. So, $x_1^*$ and $x_2^*$ are included in $\partial \hat{\mathcal{N}}$.

Since $r \leq \frac{L}{2} \cdot \min \big\{ \frac{1}{\sigma_1}, \frac{1}{\sigma_2} \big\}$, from Theorem \ref{Thm1}, we get $\partial \hat{\mathcal{M}}(x_1^*,x_2^*) = \mathcal{T}_n (L) \cup \{ x_1^*, x_2^* \} \subseteq \mathcal{I}^{\circ}(x_1^*,x_2^*)$. But $x \in \partial \mathcal{I}(x_1^*,x_2^*)$ cannot be a candidate of the minimizer because $\hat{\mathcal{N}} \subseteq \hat{\mathcal{M}}$ from Lemma \ref{lem:nested_sets2}. Similar to the proof of Theorem \ref{Thm1}, the boundary $\partial \hat{\mathcal{N}}$ can be classified into 2 disjoint types. The first type consists of points $x$ with the following property:
$\tilde{\phi}_1(x, \tilde{L}) + \tilde{\phi}_2(x, \tilde{L}) = \psi(x)$ while the second type consists of points $x$ with the following property:
$\tilde{\phi}_1(x, \tilde{L}) + \tilde{\phi}_2(x, \tilde{L}) > \psi(x)$.
Note that $\alpha_1(x)$ and $\alpha_2(x)$ are not defined if $x \in \{ x^*_1,x^*_2 \}$. 

Consider the second type. We can separate it into two different cases as follows (recall that $\cos(\tilde{\phi}_i(x, \tilde{L} )) = \frac{\sigma_i}{\tilde{L}}\|x-x_i^*\|)$: \\
(i) $\cos(\tilde{\phi}_1(x, \tilde{L})) < 1$ and $\cos(\tilde{\phi}_2(x, \tilde{L}))  < 1$. \\
(ii) $\cos(\tilde{\phi}_1(x, \tilde{L})) = 1$ or $\cos(\tilde{\phi}_2(x, \tilde{L})) = 1$. 

We will argue that the point $x$ that satisfies $\{ x: \tilde{\phi}_1(x) + \tilde{\phi}_2(x) > \psi(x) \}$ cannot be in $\partial \hat{\mathcal{N}}$. First, consider the point $x$ that satisfies the first case.
Let $\mathcal{J} = \{ x: \Vert x-x_i^* \Vert \leq \frac{1}{\sigma_i} \tilde{L}(x) \quad \text{for} \quad i = 1,2 \}$. Since the condition $\nabla f_1(x_i^*)  = 0$ with $\Vert \nabla f_i(x) \Vert \leq \tilde{L} (x)$ implies $\Vert x-x_i^* \Vert \leq \frac{1}{\sigma_i} \tilde{L}(x)$, we know that $\mathcal{I} \subseteq \mathcal{J}$. Then, since $\cos(\tilde{\phi}_1(x, \tilde{L})) < 1$ and $\cos(\tilde{\phi}_2(x, \tilde{L})) < 1$, we get $x \in \mathcal{J}^\circ$. Due to the condition $\tilde{\phi}_1(x, \tilde{L}) + \tilde{\phi}_2(x, \tilde{L}) > \psi(x)$, there exists $\epsilon > 0$ such that for all $x_0 \in \mathcal{B}_{\epsilon} (x) \subset \mathcal{J}^\circ$, we have $\tilde{\phi}_1(x_0, \tilde{L}) + \tilde{\phi}_2(x_0, \tilde{L}) > \psi(x_0)$. So, $x \notin \partial \hat{\mathcal{N}}$. Next, consider the point $x$ that satisfies the second case. From the assumption that $\hat{\mathcal{M}} \subseteq \mathcal{I}^{\circ}$ and the fact that $\mathcal{I} \subseteq \mathcal{J}$, we can also conclude that $x \notin \partial \hat{\mathcal{N}}$.

The point $x \in \partial \hat{\mathcal{N}} \setminus \{x_1^*, x_2^* \}$ must satisfies $\{ x: \tilde{\phi}_1(x, \tilde{L}) + \tilde{\phi}_2(x, \tilde{L}) = \psi(x) \}$. Using the proof similar to Lemma \ref{lem: T_n}, the set $\{ x: \tilde{\phi}_1(x, \tilde{L}) + \tilde{\phi}_2(x, \tilde{L}) = \psi(x) \}$ is equivalent to $\mathcal{T}_n(\tilde{L})$ and from the argument above, we know that $\{ x^*_1,x^*_2 \} \in \partial \hat{\mathcal{N}} $. Therefore, we conclude that if $\hat{\mathcal{M}} \subseteq \mathcal{I}^{\circ}$ and $\partial \hat{\mathcal{M}} = \mathcal{T}_n \cup \{ x^*_1, x^*_2 \}$  then $\partial \hat{\mathcal{N}} = \mathcal{T}_n(\tilde{L}) \cup \{ x^*_1, x^*_2 \}$.
\end{proof}
%\fi
%############################# arXiv: Proof of Theorem w Convex ##############################

Note that the resulting equation $\mathcal{T}_n (\tilde{L})$ may not be symmetric since $\tilde{L}$ is a function of a convex set $\mathcal{C}$. 

Examples of $\hat{\mathcal{N}}$ compared to $\hat{\mathcal{M}}$ when the convex set constraints are a circle and a box are shown in Fig. \ref{fig:convex}.

\section{Conclusions}

In this paper we studied the properties of the minimizer of the sum of strongly convex functions, in terms of the minimizers and strong convexity parameters of these functions, along with assumptions on the gradient of these functions. While identifying the region where the minimizer can lie is simple in the case of single-dimensional functions (i.e., it is given by the interval bracketed by the smallest and largest minimizers of the functions in the sum), generalizing this result to multi-dimensional functions is significantly more complicated.  Thus, we established geometric properties and necessary conditions for a given point to be a minimizer. We considered two cases: one where the gradients of the functions have to be bounded by a value $L$ at the location of the minimizer, and the other where the gradients of the functions are bounded by $L$ everywhere inside a convex set.  We used the results from the former case to provide an estimate of the region for the latter case.  The boundaries of these regions are shown in Fig.~\ref{fig:convex} (in red and dark blue).

Our work in this paper focused on identifying necessary conditions for certain points to be minimizers, and thus the regions that we have characterized are overapproximations of the true regions.  Future work will include finding sufficient conditions for given points to be a minimizers, and generalizing these regions to handle sums of multiple strongly convex functions.

%The equation of the boundary is obtained from the equality of the angle. Second, we develop the bounded gradient inequality related to the given convex set. Then we modify the boundary equation obtained from the first constraint by applying the bounded gradient inequality. The constant gradient in the first equation is replaced by the gradient which is a function of a given convex set. The resulting minimizer region of the second case is always no bigger than the region of the first case inside a given convex set. Although the shape of the minimizer region according to the first scenario is symmetric, the shape of the second one is not due to the effect of the given convex set. 

\addtolength{\textheight}{-12cm}   % This command serves to balance the column lengths

%%%%%%%%%%%%%%%%%%%%%%%%%%%%%%%%%%%%%%%%%%%%%%%%%%%%%%%%%%%%%%%%%%%%%%%%%%%%%%%%

%%%%%%%%%%%%%%%%%%%%%%%%%%%%%%%%%%%%%%%%%%%%%%%%%%%%%%%%%%%%%%%%%%%%%%%%%%%%%%%%

%%%%%%%%%%%%%%%%%%%%%%%%%%%%%%%%%%%%%%%%%%%%%%%%%%%%%%%%%%%%%%%%%%%%%%%%%%%%%%%%

%%%%%%%%%%%%%%%%%%%%%%%%%%%%%%%%%%%%%%%%%%%%%%%%%%%%%%%%%%%%%%%%%%%%%%%%%%%%%%%%
\bibliographystyle{IEEEtran}
\bibliography{refs,dist_opt_refs}

\iffalse

\fi

\end{document}